\documentclass{article}
\usepackage{graphicx}
\usepackage{amsmath}
\usepackage{amsfonts}
\usepackage{amssymb}
\usepackage{amsthm}
\usepackage{calrsfs}
\usepackage{cite}
\usepackage{tikz}
\usepackage{url}
\DeclareMathAlphabet{\pazocal}{OMS}{zplm}{m}{n}

\newcommand{\Lb}{\pazocal{L}}

\newtheorem{mydef}{Definition}[section]
\newtheorem{myex}[mydef]{Example}
\newtheorem{myrem}[mydef]{Remark}
\newtheorem{mylem}[mydef]{Lemma}
\newtheorem{myprop}[mydef]{Proposition}
\newtheorem{mytheo}[mydef]{Theorem}
\newtheorem{mycor}[mydef]{Corollary}
\date{}
\begin{document}
\title{On the characterization of Riemannian warped product Einstein metrics}
\author{Sayed Mohammad Reza Hashemi}

\maketitle

ABSTRACT. In \cite{HPW} C. He, P. Petersen and W. Wylie investigate warped product Einstein spaces using the concept of $(\lambda,m + n)$-Einstein metrics in the Riemannian setting. The authors in \cite{HPW} extend the works of \cite{CJ}, \cite{Catino} and that of \cite{KK} on $(\lambda,n+m)$-Einstein manifolds with boundary. $(\lambda,n+m)$-Einstein structure  is globally characterized in \cite{HPW}, but it is not correctly stated. In this paper I  present a series of results, including local characterizations, which are necessary for the restatement of the main result in \cite{HPW}. Finally, I restate their global characterization, i.e. \cite[Theorem 7.2]{HPW}.

\tableofcontents
\section{Introduction}
A Riemannian manifold $(M,g)$ which has no boundary point is said to be geodesically complete provided every maximal geodesic $\gamma$: $I\rightarrow M$ is defined on 
the whole $\mathbb{R}$, i.e. $I=\mathbb{R}$. Moreover a Riemannian manifold with non-empty boundary is said to be geodesically complete provided every maximal geodesic $\gamma:$ $I\rightarrow M$ maps each finite end point of $I$ to a boundary point of $\partial M$, i.e. every maximal geodesic $\gamma:$ $I\rightarrow M$ satisfies in one of the following conditions:\\
1) $\gamma:$ $[a,\infty) \rightarrow M$ then $\gamma(a)\in\partial M$\\
2) $\gamma:$ $(-\infty,b] \rightarrow M$ then $\gamma(b)\in\partial M$\\
3) $\gamma:$ $[a,b]\rightarrow M$ then $\gamma(a)$, $\gamma(b)\in\partial M$\\
4) $\gamma:$ $(-\infty,\infty) \rightarrow M$ then $\gamma(-\infty,\infty)\cap\partial M=\emptyset$

 By completeness, we mean geodesically completeness unless we specifically mention some other sense of completeness.
\begin{mydef}{$($\cite{HPW}$)$}.
\label{definition hess}
A ($\lambda$,n+m)-Einstein manifold $(M,g,f)$ is a complete Riemannian manifold $(M^n,g)$ which may have boundary together with a smooth function $f$ (here smooth means of 
class $C^2$) on $M$ satisfying
\end{mydef}
\begin{align}
 Hess f&=\frac{f}{m}(Ric-\lambda g) \label{hess} \\
 f&>0 \text{ on int}(M)\nonumber  \\
 f&=0 \text{ on }\partial M. \nonumber
\end{align}
If $m=1$ we additionally assume that $\Delta f=-\lambda f$. A $(\lambda,n+m)$-Einstein manifold $(M,g,f)$ itself is not necessarily Einstein. A simple case of $(\lambda,n+m)$-Eintstein manifolds happens when $f$ is constant. Then from Equation \eqref{hess} one obtains $Ric = \lambda g$. Moreover as $f$ is constant it can not be identically zero, hence $\partial M=\emptyset$. In \cite{HPW} this case is called a $\lambda$-Einstein manifold or a trivial $(\lambda,n+m)$-Einstein manifold.

Reminder: For the case $m=1$ see \cite{C}. In this paper and in particular in the main results, i.e. theorems \ref{improved 7.1}, \ref{aroundcritical} and \ref{Petersen22.2} we focus on $m>1$. For a Riemannian manifold $M$ without boundary compare \cite{KK}.

 Our motivation to study $(\lambda,n+m)$-Einstein manifolds is due to \cite[Proposition 5]{KK} from which it follows:
\begin{myprop}{$($\cite{HPW}, Proposition 1.1.$)$}.
\label{motivation21}
 Suppose $\lambda\in\mathbb{R}$,  $m>1$ is an integer, $(M,g)$ is a (geodesically) complete Riemannian manifold of dimension $n$ and $f\in C^\infty(M)$ is non-negative. Then $(M,g,f)$ is a $(\lambda,n+m)$-Einstein 
 manifold if and only if there is a smooth $n+m$ dimensional Riemannian warped product Einstein metric $g_E$ on $E=M^n\times F^m$ with Einstein constant $\lambda$ of the form
 $$g_E=g+f^2g_{F^m}$$
for a Riemannian Einstein manifold $(F^m,g_F)$ of dimension $m$ satisfying $Ric_F=\mu g_F$ where $\mu$ satisfies 
\begin{align}
\label{mue}
\mu=f\Delta f+(m-1)\lvert \nabla f\lvert^2+\lambda f^2
\end{align}
\end{myprop}
In the definition of a $(\lambda,n+m)$-Einstein manifold $(M,g,f)$ for the case $m=1$, the additional condition $\Delta f=-\lambda f$ is equivalent to $\mu=0$. This is necessary for the existence of $F$ because in one dimensional manifolds we must have $Ric\equiv 0$. 

If we define $u\in C^\infty(M)$ via $e^{-\frac{u}{m}}=f$ on the interior of a $(\lambda,n+m)$-Einstein manifold $(M,g,f)$, Equation \eqref{hess} takes the form
\begin{align}
 Ric_u^m=Ric+Hess u-\frac{du\otimes du}{m}=\lambda g,\label{emery}
\end{align}
$Ric_u^m$ is sometimes called the m-Bakry-Emery tensor. The so-called m-quasi-Einstein manifolds were introduced by J. Case, Y-J. Shu and G. Wei in \cite{CJ} as a triple $(M,g,u)$ consisting of a Riemannian manifold $(M,g)$ and 
a smooth function $u$ as above which satisfies \eqref{emery} for $0<m\leq\infty$ and $\lambda\in\mathbb{R}$. By the results of \cite{KK}, \cite[Theorem 2.2]{CJ} characterizes m-quasi Einstein manifolds $(M,g,u)$ as the base manifold of an Einstein warped product for which $u$ is the warping function. Additionally through \cite[Proposition 3.6]{CJ}, they proved a rigidity property for scalar curvature,  Scal, by giving lower and upper bounds respecting the sign of $\lambda$. 

A bit later, Catino in \cite{Catino} considered the following extended form of the equation \eqref{emery}
\begin{align}
\mathrm{Ric}+\mathrm{Hess}u-\varsigma du\otimes du=\lambda g\label{quasi2p}
\end{align}
for the so-called generalized quasi-Einstein manifolds. Here $u,\varsigma,\lambda$ are three smooth functions on a complete Riemannian manifold $(M^n,g)$, $n\geq3$.
Equation \eqref{quasi2p} gives out Einstein condition when $u$ and $\lambda$ are constant, and a quasi-Einstein manifold when $\varsigma$ and $\lambda$ are constant. Catino proves the following.
\begin{myprop}{$($\cite{Catino}, Theorem 1.1.$)$}.
Let $(M^n,g)$, $n\geq3$, be a generalized quasi-Einstein manifold with harmonic Weyl tensor and $W(\nabla u,.,.,.)=0$. Then around any regular point of $u$ the 
manifold $(M^n,g)$ is locally a warped product with $n-1$ dimensional Einstein fibres.
\end{myprop}
G. Catino, C. Mantegazza, L. Mazzieri and M. Rimoldi in \cite{Catino1} also consider Equation \eqref{quasi2p} where $\varsigma,\lambda\in\mathbb{R}$. They prove the following result for an arbitrary $\varsigma\in\mathbb{R}$.
\begin{myprop}{$($\cite{Catino1}, Theorem 1.1.$)$}.
Let $(M^n,g)$, $n\geq3$, be a complete locally conformally flat quasi Einstein manifold. Then the following hold:\\

(i) If $\varsigma=\frac{1}{2-n}$, then $(M^n,g)$ is globally conformally equivalent to a space form.\\

(ii) If $\varsigma\neq\frac{1}{2-n}$, then around any regular point of $u$, the manifold $(M^n,g)$ is locally a warped product with $n-1$ dimensional fibres of constant sectional curvature.
\end{myprop}
In a series of papers \cite{HPW}, \cite{HPW4}, \cite{HPW5} C. He, P. Petersen and W. Wylie investigate warped product Einstein spaces using the concept of $(\lambda,m + n)$-Einstein metrics in the Riemannian case. These metrics can also be seen as generalizations of gradient Ricci solitons, which are invariant under the Ricci flow. Our work is mostly based on \cite{HPW} where the authors extend the works of \cite{CJ}, \cite{Catino} and that of \cite{KK} on $(\lambda,n+m)$-Einstein manifolds with boundary. In \cite{HPW} setting new quantities 
$\rho:=\frac{1}{m-1}((n-1)\lambda-Scal)$ and $P=Ric-\rho g$, the authors get control on the number of corresponding eigenvalues of the tensor P with respect to $\nabla f$ (hence for the Schouten tensor and Hess $f$):
\begin{mylem}{$($\cite{HPW}, Lemma $7.1)$}.
\label{ev2}
Let $(M,g,f)$ be a Riemannian
$(\lambda,n+m)$-Einstein manifold with harmonic Weyl tensor and $W(\nabla f,Y,Z,\nabla f)=0$. Then at a point $p$ where $\nabla f\neq0$, the tensor $P$ (hence Ricci and Schouten tensors) has at most two eigenvalues. If it has two eigenvalues then one has multiplicity $1$ with eigenvector $\nabla f$, 
and the other one has multiplicity $n-1$ with vectors orthogonal to $\nabla f$. If it has one eigenvalue then $(M,g)$ is Einstein.
\end{mylem}
Then the authors define $O=\{x\in M:df(x)\neq0,\sigma_1(x)\neq\sigma_2(x)\}$ where $\sigma_1$ and $\sigma_2$ are the eigenvalues of the Schouten tensor. If $(M,g)$ is Einstein, then there is only one eigenvalue, i.e. $\sigma_1=\sigma_2$. In this case Hess$f$ is proportional to the metric, cf. Example  \ref{einstein case} for the relevant results. Using Lemma \ref{ev2} they decompose the metric $g$ in a neighborhood of a point $p\in O$ into a warped product $g=dt^2+u^2(t)g_N$, further they show that $f=f(t)$:
\begin{mytheo}{$($\cite{HPW}, Theorem $7.1)$}.
\label{7.222}
Suppose $m>1$ and $(M,g,f)$ is a $(\lambda,n+m)$-Einstein metric with harmonic Weyl tensor and $W(\nabla f,.,.,\nabla f)=0$ in an open set containing $p\in O$. Then 
 $$g=dt^2+u^2(t)g_N$$
 $$ f= f(t)$$
 in a neighborhood of $p$, where $g_N$ is an Einstein metric. Moreover if the metric is locally conformally flat around $p$, then $N$ is a space of constant sectional curvature.
\end{mytheo}
Then the authors globally characterize $(\lambda,n+m)$-Einstein metrics in the Riemannian case as their main result:
\begin{mytheo}{$($\cite{HPW}, Theorem $7.2)$}.
\label{correction}
Let $m>1$ and suppose that $(M,g)$ is a complete, simply connected Riemannian manifold and has harmonic Weyl tensor and $W(\nabla f,.,.,\nabla f)=0$, then $(M,g,f)$ is a non-trivial 
	$(\lambda,n+m)$-Einstein metric if and only if it is of the form
	$$g=dt^2+u^2(t)g_N$$
	$$f=f(t)$$
	where $g_N$ is an Einstein metric. Moreover, if $\lambda\geq0$ then $(N,g_N)$ has non-negative Ricci curvature, and if it is Ricci flat, then $u$ is a constant, i.e, $(M,g)$ is a Riemannian product.
\end{mytheo}
Theorem \ref{correction} is not exactly stated. For example, there is a little observation in connecting Theorem \ref{7.222} with Theorem \ref{correction} which is missing in \cite{HPW}. In fact, Theorem \ref{7.222} is a classification around points $p\in O$, while making a global characterization it may happen that in a neighborhood of a regular point there are some points at which the Schouten tensor has only one eigenvalue, cf. Lemma \ref{ev2}. We go through this observation in the proof of Theorem \ref{improved 7.1}. More prerequisites for the restatement of Theorem \ref{correction} are listed right before Theorem \ref{Petersen22.2}. 

As the main results of this paper, we first formulate relation between the warping function of a local warped product form of $g$ and the function $f$ of a $(\lambda,n+m)$-Einstein metric $(M,g,f)$. Then we characterize $(\lambda,n+m)$-Einstein metrics around regular points of $f$ where $g$ has harmonic Weyl tensor and satisfies $W(\nabla f,.,.,\nabla f)=0$. In addition we show that under the conditions of harmonicity of the Weyl tensor and $W(\nabla f,.,.,\nabla f)=0$, critical points of $f$ in a  $(\lambda,n+m)$-Einstein metric $(M,g,f)$ are isolated; consequently we characterize the $(\lambda,n+m)$-Einstein metric if $g$ is locally conformally flat around critical points of $f$. Finally we restate the global characterization in Theorem \ref{correction}. 

This paper is organized in the following way. In the second section we bring some examples. In the third section, first we obtain the formula describing the relation between the warping function of a one-dimensional basis warped product and the function $f$ of a $(\lambda,n+m)$-Einstein structure $(M,g,f)$ under some specific conditions. Then, we locally characterize $(\lambda,n+m)$-Einstein manifolds around regular points of $f$. In section 4 we show that critical points of $f$ in a triple $(M,g,f)$ satisfying Equation \eqref{hess} of a non-trivial $(\lambda,n+m)$-Einstein metric with harmonic Weyl tensor and $W(\nabla f,.,.,\nabla f)=0$ are isolated. Then we characterize a triple $(M,g,f)$ which is conformally flat and satisfies Equation \eqref{hess} of a non-trivial $(\lambda,n+m)$-Einstein metric around a critical point of $f$. We close it in section 5 by explaining the unnecessary and missing  properties in the formulation of the global result \cite[Theorem 7.2]{HPW} and restate it.
\section{Examples}
 Classical solutions like $u(t)=t, sin(t), e^t$, $cosh(t)$ give us one dimensional basis warped product $(\lambda,n+m)$-Einstein metrics for appropriate choices of $f(t)$, $k$, $\lambda$ 
and $m$. In particular $u(t)=e^t$ with $\lambda=-n-m+1,\enskip k=0$ and $f(t)=ae^t,\enskip a\in {\mathbb {R}}^+$ provides us with an example where $(M,g)$ is Einstein with normalized scalar 
curvature $k_g=-1$ (by \cite[Lemma 2.5.(22)]{KR3}). Here the manifold $(M,g)$ is without boundary as the function $f$ is always positive. This example is not interesting by Proposition \ref{einstein case} and \cite[page 434]{KR1}. As a classical example with non-empty boundary we have:
\begin{myex}
\label{examp2}
As a classical example let 
$g= dt^2+cosh^2(t)g_{H^{n-1}(-1)}$ on $M=[0,\infty)\times H^{n-1}(-1)$
where $H^{n-1}(-1)$ is the hyperbolic space with standard metric, further let $\lambda=1-n-m$ and  $f(t)=\sinh(t)$. Then we obtain a non-trivial $(\lambda,n+m)$-Einstein manifold with the boundary the slice  $\{t_0=0\}\times H^{n-1}(-1)$. 
\end{myex}
 We may recall the following classification of Riemannian $(\lambda,n+m)$-Einstein metrics which are also Einstein. 
\begin{myex}{$($\cite{HPW}, Proposition $3.1)$}.
\label{einstein case}
Let $n\geq2$ and let $(M^n,g,f)$ be a non-trivial $(\lambda,n+m)$-Einstein manifold which in addition is $\rho$-Einstein. Then it is isometric to one of the examples in the following table for
$\bar{k}=\frac{\lambda-\rho}{m}$.

\begin{center}
\begin{tabular}{ |c|c|c|c| }
 \hline
 & $\lambda > 0$ & $\lambda = 0$ &$ \lambda < 0$\\ \hline
 {$\mu > 0$} &$ D^n$ & [0,$\infty$)$ \times F$ & [0,$\infty$) $\times N$\\ 
 & g=$dt^2+\sqrt{\bar{k}} sin^2(\sqrt{\bar{k}}t) g_{S^{n-1}}$& g=$dt^2+g_F $& g=$dt^2+\sqrt{-\bar{k}} 
 cosh^2(\sqrt{-\bar{k}}t)g_N$\\
&  f(t)=C $cos(\sqrt{\bar{k}}t)$& $f(t)=C t$ & f(t)=C $\sinh(\sqrt{-\bar{k}}t)$\\ \hline
 {$\mu$ = 0} & None & None & (-$\infty$, $\infty$) $\times F$\\
 & & & g=$dt^2+e^{2\sqrt{-\bar{k}}t}g_F$\\
 & & & $ f(t)=C e^{\sqrt{-\bar{k}}t}$\\ \hline
 {$\mu < 0$} & None & None & $H^n$\\
 & & & g=$dt^2+\sqrt{-\bar{k}} \sinh^2(\sqrt{-\bar{k}}t)g_{S^{n-1}}$\\
 & & & f(t)=C $\cosh(\sqrt{-\bar{k}}t)$\\
 \hline
\end{tabular}
\end{center}
In the table above, $S^{n-1}$ denotes a round sphere, $F$ is Ricci flat and $N$ denotes an Einstein metric with negative Ricci curvature and $C\in\mathbb{R}^+$ is arbitrary.
\end{myex}
For more examples see \cite[Example 9.118(a)]{B} as well as Section 3 of \cite{HPW}. For example,  \cite[Example 3.5]{HPW} indicates a local Riemannian $(\lambda,3+m)$-Einstein structure which is not locally conformally flat. 
\section{Local characterization of Riemannian $(\lambda,n+m)$-Einstein metrics $(M,g,f)$ around regular points of $f$}
One dimensional Riemannian $(\lambda,1+m)$-Einstein manifolds are classified in \cite[Example 3.1]{HPW}. Hence, we only consider $dim(M)=n>1$ unless the opposite is stated. First,  we characterize $(\lambda,n+m)$-Einstein manifolds $(M,g,f)$ where $g$ is a one dimensional-basis warped product with Einstein fibre. Through this characterization, in fact we want to formulate the relation between the warping function of a local warped product form of $g$ and the function $f$.
\begin{myprop}
\label{equ}
Let $\lambda\in\mathbb{R}$, $m\geq1$, $n>1$ integers and $g_N$ a Riemannian Einstein metric say with normalized scalar curvature $\varrho_N=k\in\mathbb{R}$, i.e. $Ric_N = k(n-2)g_N$, on an $(n-1)$-dimensional
manifold $N$ and $g=dt^2 + u^2(t)g_N$ a warped product metric on $M = I\times N$ with an interval $I\subset\mathbb{R}$. In addition suppose $f=f(t)$ is a smooth non-negative  function on $I$. Then $(M=I\times N,g,f)$ satisfies Equation \eqref{hess} of a $(\lambda,n+m)$-Einstein manifold if and only if the following conditions hold\\\\
 1. On $int(M)$
\begin{align}
&f'm\frac{u'}{u}+\{\lambda-\frac{(n-2)k-(n-2)u'^2-uu''}{u^2}\}f=0 \label{equ1}\\\nonumber\\
&\lambda^2 u^4-2(n-2)k\lambda u^2+(m+2(n-2))\lambda u^2u'^2+(2+m)\lambda u^3u''\nonumber\\
+&(n-2)^2k^2-(2(n-2)+m)(n-2)ku'^2-(2+m)(n-2)kuu''\label{equuu}\\
+&(n-2)(m+n-2)u'^4+(2(n-2)+m)uu'^2u''+(1+m)u^2u''^2\nonumber\\
-&mu^2u'u'''=0\nonumber
\end{align}
2. On $\partial M$ 
\begin{align}
\label{equ0}
 f''(t)=u'(t)=0.
\end{align}
\end{myprop}
\begin{proof}
As $(M,g)$ may have boundary, first we consider $g=dt^2+u^2(t)g_N$ on $int(M=I\times N)$ on which $f>0$.   
Using the relations $\nabla_X\partial_t=\frac{u'}{u}X$ and $\nabla_{\partial_t}\partial_t=0$ one obtains
\begin{align}
\nabla^2f(\partial_t,\partial_t)=f''g(\partial_t,\partial_t)=f''\label{equ3}\\
\nabla^2f(X,X)=\frac{f'u'}{u}g(X,X).\label{30}
\end{align}
Also, from \cite[Lemma 2.5 ]{KR3} we have the following equations where $k$ signifies the normalized scalar curvature of $g_N$
\begin{align}
Ric(\partial_t,\partial_t&)=-(n-1)\frac{u''}{u}g(\partial_t,\partial_t)=-(n-1)\frac{u''}{u}\label{equ5}
\end{align}
and
\begin{align}
Ric(X,X)&=Ric_N(X,X)-\frac{1}{u^2}[(n-2)u'^2+uu'']g(X,X)\nonumber\\
&=k(n-2)g_N(X,X)-\frac{1}{u^2}[(n-2)u'^2+uu'']g(X,X)\label{equ66}
\end{align}
where $X$ denotes any unit tangent vector orthogonal to $\partial_t$. Let $(M,g,f)$ satisfies Equation \eqref{hess} of a $(\lambda,n+m)$-Einstein manifold on $int(M)$. Tangent vectors to $M$ can be divided into the category of those tangent in the direction of the first factor of $I\times N$, i.e. scalar multiplications of $\partial_t$, and into the category of tangent vectors to 
$N$. Tensorial Equation \eqref{hess} on vectors of the first category,
\begin{align}
\nabla^2 f(\partial_t,\partial_t)=\frac{f}{m}(Ric(\partial_t,\partial_t)-\lambda g(\partial_t,\partial_t)),\label{equ29}
\end{align}
via \eqref{equ3} $\&$ \eqref{equ5} becomes
\begin{align}
f''+(\frac{(n-1)}{m}\frac{u''}{u}+\frac{\lambda}{m})f=0,\label{equ31}
\end{align}
the so-called first necessary condition. Now let's pay attention to evaluation on the second category of vector fields
\begin{align}
\label{equ35}
 \nabla^2f(X,Y)=\frac{f}{m}(Ric(X,Y)-\lambda g(X,Y)) \enskip\enskip\enskip X,Y\perp \frac{\partial}{\partial t},
\end{align}
which by \eqref{30} $\&$ \eqref{equ66} forms the so-called second necessary condition,
\begin{align}
f'm\frac{u'}{u}+\{\lambda-\frac{(n-2)k-(n-2)u'^2-uu''}{u^2}\}f=0.\label{equ32}
\end{align}
Note that \eqref{equ1} is actually the second necessary condition \eqref{equ32}. Furthermore, through the functions 
\begin{align*}
 &a(t)=\frac{(n-1)}{m}\frac{u''}{u}+\frac{\lambda}{m}\\
 &b(t)=m\frac{u'}{u}\\
 &c(t)=\lambda-\frac{(n-2)k-(n-2)u'^2-uu'')}{u^2}
\end{align*}
we can denote the necessary conditions \eqref{equ31} and \eqref{equ32} 
respectively by
\begin{align}
 &f''=-af\label{equ8}\\
 &f'=-\frac{c}{b}f.\label{equ9}
\end{align}
We still need to show that the second claim \eqref{equuu} is satisfied: We differentiate \eqref{equ9} and then compare it to \eqref{equ8} which gives us
\begin{align}
f''=-af=(-\frac{c}{b}f)'=-(\frac{c}{b})'f-(\frac{c}{b})f'.\label{equ99}
\end{align}
By further application of \eqref{equ9} we obtain
\begin{align}
\label{equ110}
 -af=-(\frac{c}{b})'f+(\frac{c}{b})^2f
\end{align}
which reduces to
\begin{align}
-a=-(\frac{c}{b})'+(\frac{c}{b})^2.\label{equ100} 
\end{align}
We rewrite \eqref{equ100} in terms of $u(t)$, $u'(t)$ and $u''(t)$ as   
\begin{align}
&-\Big(\frac{(n-1)}{m}\frac{u''}{u}+\frac{\lambda}{m}\Big)=-\Big(\{\lambda-\frac{(n-2)k-(n-2)u'^2-uu''}{u^2}\}\frac{u}{mu'}\Big)'\nonumber\\
&+\Big(\lambda-\frac{(n-2)k-(n-2)u'^2-uu''}{u^2}\Big)^2\frac{u^2}{m^2u'^2}\label{equ4s}
\end{align}
which after simplification gives us \eqref{equuu}.\\ 

On $int(M)$ for the converse, suppose that we have 
$f=f(t)$ and $g=dt^2+u^2(t)g_N$ where $g_N$ is Einstein with normalized scalar curvature $\varrho_N=k$. 
Additionally, assume Equation \eqref{equ1} for some $\lambda\in\mathbb{R}$ holds. This says that the so-called second necessary condition \eqref{equ32} holds. 

It remains to show that  tensorial Equation \eqref{hess} is also satisfied with
vectors tangent on the first factor of $I\times N$ to prove the triple $(M,g,f)$ satisfies Equation \eqref{hess} of $(\lambda,n+m)$-Einstein manifolds. To this end, we utilize Equation \eqref{equuu} which is the simplification of \eqref{equ4s}. On the other hand by the labels $a(t)$, $b(t)$ and $c(t)$ Equation \eqref{equ4s} can be written 
in the form of \eqref{equ100} which after multiplication by $f$ gives out \eqref{equ110} (note that $f>0$ on $int(M)$). From Equation \eqref{equ110} together with Equation \eqref{equ9}, which is actually \eqref{equ1} using the labels $a(t)$, $b(t)$ and $c(t)$, we see that Equation \eqref{equ99} holds. On the other hand, the derivative of \eqref{equ99} comparing with \eqref{equ9} gives out Equation \eqref{equ8} which is equivalent to the so-called first necessary condition \eqref{equ31}. So, for all vector fields tangent to $int(M)$ the tensorial Equation \eqref{hess} of a 
$(\lambda,n+m)$-Einstein manifold holds.\\

Secondly let's investigate the characterization on $\partial M$: Since the Hessian tensor $\nabla^2 f$ vanishes on 
the boundary $\partial M$, equations \eqref{equ3} and \eqref{30} imply
\begin{equation}
 f''(t)=0
\end{equation}
and
\begin{equation}
f'(t)u'(t)=0
\end{equation}
from which the latter together with Proposition \ref{constantboundary} (which says $f'(t)\neq0$ on $\partial  M$) implies $u'(t)=0$. Therefore on 
$\partial M$ one obtains Equation \eqref{equ0}. 

Conversely, as $f$ vanishes on $\partial M$ the right-hand side of Equation \eqref{hess} becomes zero. Thus the left-hand side of Equation \eqref{hess} must also 
vanish on it. On the other hand by Equation \eqref{equ0} the Hessian tensor $\nabla^2f$ identically vanishes on 
$\partial M$, hence Equation \eqref{hess} holds in this case as well. 

Therefore in each case the tensorial Equation \eqref{hess} of a $(\lambda,n+m)$-Einstein manifold holds on $M$ .
\end{proof}
\begin{myrem}
 In Proposition \ref{equ} for a positive solution $u(t)$ of \eqref{equuu} on $int(M)$, Equation \eqref{equ1} expresses $f(t)$ in terms of $u(t)$ by 
 $f(t)=e^{-\int^t_0\frac{c(s)}{b(s)}ds}f(0)$ as far as $u'\neq0$. On $\partial M$ the functions $f(t)$ and $u(t)$ satisfy Equation \eqref{equ0}.
\end{myrem}
\begin{myrem}
The procedure through equations \eqref{equ8} to \eqref{equ100} in the proof of Proposition \ref{equ} shows that if $f$ satisfies \eqref{equ1}, i.e. the so-called second necessary condition  
, then it also satisfies the so-called first necessary condition \eqref{equ31} of a $(\lambda,n+m)$-Einstein 
manifold under condition \eqref{equuu} which is equivalent to \eqref{equ100} and hence, the domain of solution of the differential equations system of the $(\lambda,n+m)$-Einstein manifold consisting of \eqref{equ31} and \eqref{equ32} is not empty for $f$.
\end{myrem}
Before we characterize locally $(\lambda,n+m)$-Einstein manifold $(M,g,f)$ around regular points of $f$, we need to remind a pair of statements:
\begin{myrem}{$($\cite{HPW}, Remark $7.3)$}.
\label{p6}
$\nabla f$ is an eigenfield for the tensor $P$ (or for the Schouten tensor) if and only if $\nabla f$ is an eigenfield for the tensor Hess $f$. If this holds then $|\nabla f|^2$ is 
constant on the connected components of the level sets of $f$. Because for any $X\perp \nabla f$
 $$\nabla_X|\nabla f|^2=2Hess f(\nabla f,X)=2\mu_1g(\nabla f, X)=0$$
in particular, the connected components of the regular levels sets for $f$ form
a Riemannian foliation of an open subset of $M$.
\end{myrem}
\begin{myprop}{$($\cite{HPW}, Proposition 2.2.$)$}.
 \label{constantboundary}
 On the boundary $\partial M$ of a Riemannian $(\lambda,n+m)$-Einstein manifold $(M,g,f)$ we always have $\nabla f\neq0$.
\end{myprop}
Now we generalize Theorem \ref{7.222} to a local characterization which assumes the weaker condition $\nabla f|_p\neq0$ instead of the stronger assumption $p\in O$. 
\begin{mytheo}
\label{improved 7.1}
 Let $m>1$, $\lambda\in\mathbb{R}$ and $(M,g)$ be a Riemannian manifold with a smooth function $f$ defined on $M$. Then the following conditions are equivalent:\\  
 
 1) $(M,g,f)$ satisfies Equation \eqref{hess} of a non-trivial $(\lambda,n+m)$-Einstein metric with harmonic Weyl tensor and $W(\nabla f,.,.,\nabla f)=0$ in a neighborhood of $p\in M$ with $\nabla f|_p\neq0$.\\ 

2) (a) Case $p\in int(M)$: There exist local coordinates $(t,t_1,...,t_{n-1})$ with $t\in(-\varepsilon,\varepsilon)$ in a neighborhood of $p\in M$ and an Einstein Riemannian hypersurface 
$(N,g_N=g_N(t_1,...,t_{n-1}))$ of $(M,g)$ with normalized scalar curvature $\varrho_N=k$ and a function $u=u(t)>0$, in addition $f=f(t)>0$ satisfying \eqref{equ1} and \eqref{equuu} in Proposition \ref{equ} such that\\\\
$I)$ $\enskip g(\partial_t,\partial_t)=1$\\
$II)$ $g(\partial_t,\partial_{t_i})=0,\enskip\enskip\text{for}\enskip i=1,...,n-1$\\
$III)$ $g(\partial_{t_i},\partial_{t_j})=u^2(t)g_N(\partial_{t_i},\partial_{t_j})(t_1,...,t_{n-1})\enskip\text\enskip i,j=1,...,n-1.$\\

(b) Case $p\in\partial M$:  
There exist local coordinates $(t,t_1,...,t_{n-1})$ with $t\in[0,\varepsilon)$ in a neighborhood of $p$ and an Einstein Riemannian hypersurface $(N,g_N=g_N(t_1,...,t_{n-1}))$ of $(M,g)$ with normalized scalar curvature $\varrho_N=k$ and a function $u=u(t)>0$, in addition  $f(t)>0$ for all $t\in(0,\varepsilon)$  satisfying \eqref{equ1} $\&$ \eqref{equuu} as well as $f(0)=0$ satisfying \eqref{equ0} at $t=0$ such  that the conditions $I$, $II$, $III$ in $(a)$ holds.\\\\
Any case of 2) implies that $g=dt^2+u^2(t)g_N$ around $p$.
\end{mytheo}
\begin{proof}
1) $\Rightarrow$ 2): By Lemma \ref{ev2} the Schouten tensor has at most two eigenvalues $\sigma_1$ and $\sigma_2$ at each point in a neighborhood $\widetilde{\mathcal{U}}$ of $p$ consisting of regular points of $f$. If $\sigma_1=\sigma_2$ in an open  subset $U$ of $\widetilde{\mathcal{U}}$ then via Schur's lemma $g$ is Einstein on $U$. In addition, the derivative of $\sigma=\sigma_1-\sigma_2=0$ vanishes on $U$. By analyticity of $g$, see \cite[Proposition 2.4]{HPW}, $d\sigma$ and hence $\sigma$ vanish on $\widetilde{\mathcal{U}}$. Therefore it would be Einstein, i.e. $\sigma_1=\sigma_2$, on the whole $\widetilde{\mathcal{U}}$. In this case, using Example \ref{einstein case} we see that the conclusion of this theorem is satisfied. Therefore we can assume that the open set $O\cap\widetilde{\mathcal{U}}$ is dense in $\widetilde{\mathcal{U}}$.

For more convenience we first suppose that $p\in int(M)$. Without loss of generality assume $p\in O\cap\widetilde{\mathcal{U}}$ (otherwise we may start by a point $p_1\in O\cap\widetilde{\mathcal{U}}$ and through the same procedure as follows we get the same result on $\widetilde{\mathcal{U}}\ni p$), thus the Schouten tensor $S$ has two different eigenvalues $\sigma_1$ and $\sigma_2$ in a neighborhood $\mathcal{U}$ of $p$ in $\widetilde{\mathcal{U}}$. By Lemma \ref{ev2} we know also that the dimension of the eigenspace corresponding to $\sigma_2$ is bigger than one when dim$M>1$, hence \cite[16.11(iii)]{B} proves that $\sigma_2$ is locally constant on the level sets of $f$ in $\mathcal{U}$. 

As the Schouten tensor $S$ has two distinct eigenvalues in $\mathcal{U}$, via the relation $Hess f=\frac{f}{m}(S+(\frac{Scal}{2(n-1)}-\lambda)g)$ it follows that Hess$f$ has also two 
distinct eigenvalues in $\mathcal{U}$, call them $\mu_1$ and $\mu_2$, where the eigenspaces for $\mu_i$ correspond to eigenspaces for $\sigma_i$ by 
\begin{align}
\label{regular151}
 \mu_i=\frac{f}{m}(\sigma_i+\frac{Scal}{2(n-1)}-\lambda)\enskip\enskip i=1,2.
\end{align}
We already know by Remark \ref{p6} that $|\nabla f|$ is locally constant on the level sets of $f$ in $\mathcal{U}$ which in turn concludes that $\mu_1$ is also locally constant on the level sets of 
$f$ in $\mathcal{U}$. In more details, from
$$\frac{1}{2}D_{\nabla f}|\nabla f|^2=Hess f(\nabla f,\nabla f)=\mu_1|\nabla f|^2$$
we get
$$\mu_1=\frac{1}{2}\frac{1}{|\nabla f|^2}D_{\nabla f}|\nabla f|^2$$
and hence
\begin{align}
 D_X\mu_1=\frac{1}{2}\frac{1}{|\nabla f|^2}D_XD_{\nabla f}|\nabla f|^2=\frac{1}{2}\frac{1}{|\nabla f|^2}D_{\nabla f}D_X|\nabla f|^2=0,\enskip X\perp \nabla f.\label{why isotropic}
\end{align}
Moreover if $X\perp \nabla f$ then
$$D_X\rho=\frac{2}{f}P(\nabla f,X)=0$$
expressing that $\rho$ and hence the scalar curvature $Scal$ are locally constant on the level sets of $f$, hence by Equation \eqref{regular151} $\sigma_1$ and $\mu_2$ are  
locally constant on the level sets of $f$. So $|\nabla f|^2$, $\mu_1$ and $\mu_2$ are all locally constant on the level sets of $f$ in $\mathcal{U}$.\\

Let $c:=f(p)$ and $N\subset\mathcal{U}$ be the connected component of $f^{-1}(c)$ containing $p$ in $\mathcal{U}$. Since $|\nabla f|\neq0$ on $N$ it follows that $(N,g_N)$ is a Riemannian hypersurface of $(M,g)$. 
One can choose a coordinate chart $(t_1,...,t_{n-1})$ on the level hypersurface $N$. We are interested to extend this chart to a neighborhood of $p$ in $M$ using $f$. For that purpose, we note that as 
the norm $|\nabla f|$ is locally constant on the level sets of $f$ in $\mathcal{U}$ it may be considered as a function of $f$, i.e. $|\nabla f|=|\nabla f|(f)$. It follows that
\begin{align}
\label{localclosed}
d(\frac{1}{|\nabla f|}df)=d(\frac{1}{|\nabla f|})\wedge df=-\frac{d|\nabla f|}{|\nabla f|^2}df\wedge df=0
\end{align}
meaning that $\frac{1}{|\nabla f|}df$ is locally closed and hence exact in $\mathcal{U}$. Therefore, there exists a smooth function $t$ on $\mathcal{U}$ such that 
\begin{align}
 dt=\frac{1}{|\nabla f|}df\enskip\text{with }t=t(f)=\int^{f}_{N}\frac{df}{|\nabla f|}\label{distance0}
\end{align}
respectively $|\nabla t|=1$.
Additionally, the symmetry of Hess $t$ together with the equation $|\nabla t|=1$ implies
$$Hess t(\nabla t,X)=g(\nabla_{\nabla t}\nabla t,X)=g(\nabla_X\nabla t,\nabla t)=\frac{1}{2}\nabla_X|\nabla t|^2=0;\enskip X\in TM$$
from which by the non-degeneracy property of $g$ it follows that $\nabla_{\nabla t}\nabla t=0$. Accordingly the trajectories of 
$\frac{\nabla f}{|\nabla f|}$, i.e. integral curves of $\nabla t$, are geodesics which are normal to level sets of $f$ in $\mathcal{U}$.

Consequently we may extend the coordinate chart on $N$ to geodesic parallel coordinates $(t,t_1,...,t_{n-1})$ in a neighborhood of $p$ satisfying:\\
- the $t$-lines are geodesics with $t$ as arc length.\\
 -$\frac{\partial}{\partial t}$ is orthogonal to every set $\{(t,t_1,...,t_{n-1})|\enskip t=constant\}$, i.e. $g(\frac{\partial}{\partial t}, \frac{\partial}{\partial t_i})=0,
 \enskip i=1,...,n-1$.\\
 This shows that the different $t$-levels are parallel to each other and the distance between them equals the difference of $t$-values.
 
Now consider the $f$-levels $\{q|\enskip f(q)=constant\}$ where the $t$-level containing $p$, for which $t=0$, coincides with $f^{-1}(c)$. 
As $|\nabla f|$ is constant along the level sets of $f$ in $\mathcal{U}$, see Remark \ref{p6}, they are also parallel to each other. Therefore, the $t$-levels coincide with the $f$-levels 
and we can consider $f$ as a function of $t$ alone: 
\begin{align}
 f(t,x)=f(t)\enskip\text{and}\enskip \nabla f(t,x)=f'(t)\frac{\partial}{\partial t}\label{alone22}
\end{align}
where $f>0$ because $M$ is boundaryless. Since $\mu_1$ and $\mu_2$ are locally constant on the level sets of $f$ in $\mathcal{U}$ they may also be considered 
as functions of $f$ by which \eqref{alone22} it follows then that they are functions of $t$, e.g. $\mu_2(t)=(\mu_2of)(t)$, which would be given then by \eqref{equ3} $\&$ \eqref{30}.

From $dt=\frac{1}{|\nabla f|}df$, i.e. $t$ a distance function, it follows that the metric $g$ in $\mathcal{U}\ni p$ can be decomposed into 
\begin{align}
 g=\frac{1}{|\nabla f|^2} df\otimes df+g_f\label{metric01}
\end{align}
where $g_f$ represents a Riemannian metric on a level set of $f$ in $\mathcal{U}$ with tangent space orthogonal complement to the space generated by the unit normal vector field 
$\frac{\nabla f}{|\nabla f|}$. By \eqref{metric01} and the fact that the eigenvalue $\mu_1$ of Hess$f$ corresponds to $\nabla f$ and $\mu_2$ corresponds to vector fields orthogonal to it, we obtain  
\begin{align}
 Hessf=\mu_1.\frac{1}{|\nabla f|^2}df\otimes df+\mu_2g_f.\label{decompose01}
\end{align}
From Equation \eqref{decompose01} one obtains 
$\Lb_{\nabla f}g_f=2Hessf|_{g_f}=2\mu_2g_f$ by which Equation \eqref{metric01} gives out a local warped structure in $\mathcal{U}$ in terms of $t$: 

We replace the first term on the line element \eqref{metric01} by $\frac{df}{|\nabla f|}=dt$ giving 
\begin{align}
 g=dt^2+g_t\label{formone}
\end{align}
where $t=0$ corresponds to $N\subset f^{-1}(c)$. Now we work on $g_t$ to acquire the desired structure. Let $X$ be a lift of a vector field on $N$, then $g(\partial t,X)=0$ by the Gauss-Lemma. 
Also for vectors $X_1$, $X_2$ tangent to $N$ at 
$x_0$ let $X_i(t)=d$ $exp(t,x_0)(X_i),i=1,2$ then
\begin {align}
&\frac{d}{dt}|_{t=s}g(X_1,X_2)(t)=\Lb_{\partial t}g(X_1,X_2)(s)=\frac{1}{f'(s)}\Lb_{\nabla f}g(X_1,X_2)(s)\nonumber\\
 &=\frac{2}{f'(s)}Hessf(X_1,X_2)(s)=\frac{2}{f'(s)}\mu_2(s) g(X_1,X_2)(s)\label{independent01}
\end {align}
where $\Lb_Zg(X,Y)=g(\nabla_XZ,Y)+g(X,\nabla_YZ)$ is the Lie derivative of the metric in direction of the vector field $Z$. By an integration step from \eqref{independent01} we obtain
$$g_t=\Big(e^{\int^t_0\frac{1}{f'(s)}\mu_2(s)ds}\Big)^2g_{N\subset f^{-1}(c)}.$$  
Therefore we may write \eqref{formone} as the warped structure 
\begin{align}
 g=dt^2+u^2(t)g_N, \enskip t\in(-\varepsilon,\varepsilon)\label{warped01}
\end{align}
where
\begin{align}
 u(t)=e^{\int^t_0\frac{1}{f'(s)}\mu_2(s)ds}.\label{power01}
\end{align}
To confirm \eqref{warped01} is a warped product metric it remains to show that $g_N$ is independent of $t$, and also, is non-degenerate. For $X_1$ and $X_2$ as above, the mapping 
$t\longmapsto (u(t))^{-2}g(X_1,X_2)(t)$ satisfies the differential equation
\begin{align}
 &(\frac{g(X_1,X_2)}{u^2})'(t)=\frac{\frac{d}{ds}|_{s=t}g(X_1,X_2)(s)}{u^2(t)}-\frac{2u'(t)}{u^3(t)}g(X_1,X_2)(t)=0
\end{align}
expressing $g_N(X_1,X_2)=\frac{g(X_1,X_2)}{u^2(t)}$ is independent of the coordinate function $t$ and hence one may introduce this expression as $g_N(t_1,...,t_{n-1})$. To see $g_N$ is 
non-degenerate, suppose $g_N(X,Y)=0$ for some $X$ and all $Y$ tangent to $N$. On the other hand by \eqref{warped01} we have $g(X,\partial t)=0$. As the metric $g$ is non-degenerate we 
obtain $X=0$. By Proposition \ref{equ} the equations \eqref{equ1} and \eqref{equuu} are satisfied for a boundaryless manifold.

Now suppose that $p\in\partial M$. Then through the same discussion as above and noting that due to $N\subseteq\partial M$ the trajectory geodesics of $\nabla f$ starting at $N$ point only to one side, e.g. to its positive side, it follows that there exist geodesic parallel  coordinates $(t,t_1,t_2,...,t_{n-1})$ in a neighborhood of $p$ with $t\in[0,\epsilon)$, $t(p)=0$ for which the result holds. Moreover the conditions \eqref{equ1}, \eqref{equuu}, \eqref{equ0} in Proposition \ref{equ} hold on $\mathcal{U}$.

Considering the way of defining the function $t$ at above, we see that $t$ is smooth as far as $\nabla f\neq0$, cf. Equation \eqref{distance0}, including the points at which $\sigma_1=\sigma_2$.

By smoothness of the metric $g$ and $t$ it follows that we have a warped product, where $g_N$ is an Einstein Riemannian hypersurface of $(M,g)$, along all of $t$ as long as $\nabla f\neq0$, therefore on the whole $\widetilde{\mathcal{U}}$. In parallel, to extend the relations $f=f(t)$, and \eqref{equ1},  \eqref{equuu}, \eqref{equ0} as far as $\nabla f\neq0$ (so on the whole $\widetilde{\mathcal{U}}$) we use again a similar discussion as above.

Now we see that any metric of this form whose Schouten and hence the Ricci tensor has two distinct eigenvalues must have $g_N$ Einstein. By Lemma \cite[Lemma 2.5.(21)]{KR3}
the first eigenvalue of the Ricci tensor would be $\gamma_1=-(n-1)\frac{u''(t)}{u(t)}$. As the second eigenvalue $\gamma_2$ corresponds to vectors $X,Y\perp\partial_t$, using Lemma \cite[Lemma 2.5.(19)]{KR3} we have
$$Ric(X,Y)=\gamma_2g(X,Y)=Ric_N(X,Y)-\frac{1}{u^2}[(n-2)u'^2+uu'']g(X,Y)$$
thus
$$Ric_N(X,Y)=(\gamma_2+\frac{1}{u^2}[(n-2)u'^2+uu''])g(X,Y)$$
implying that $g_N$ is Einstein. Therefore the conclusion of the theorem is satisfied on the whole $\widetilde{\mathcal{U}}$.\\ 

2) $\Rightarrow$ 1): Suppose by contradiction that $\nabla f(p)=0$.
If it is Case $(b)$ then by Proposition \ref{constantboundary} it is a contradiction. If it Case $(a)$, then we may consider geodesic polar coordinates with origin at $p$, hence it must satisfy $u(a)=0$, $a=t(p)$. This is a contradiction with the assumption. 

Let the metric be isometrically $g=dt^2+u^2(t)g_N$ where $g_N$ is an Einstein hypersurface with $\varrho_N=k$, and equations \eqref{equ1}, \eqref{equuu}, \eqref{equ0} be satisfied by $u=u(t)>0$ and  $f=f(t)\geq0$. Then Proposition \ref{equ} implies that $(M,g,f)$ satisfies Equation \eqref{hess} of $(\lambda,n+m)$-Einstein manifolds around $p$.

By \cite[16.26(i)]{B} the manifold $(M,g)$ has harmonic Weyl tensor and satisfies $W(\nabla f, ., ., \nabla f)=0$ around $p$.
\end{proof}
In the Riemannian case, Catino proved existence of a local warped product metric with $(n-1)$ dimensional Einstein fibre around regular points of $f$ in a $(\lambda,n+m)$-Einstein manifold $(M,g,f)$, see \cite[Theorem 1.1]{Catino}. In this Theorem Catino assumes harmonicity of the Weyl tensor $W$ and $W(\nabla f,.,.,.)=0$. The assumption $W(\nabla f,.,.,.)=0$ by Catino is stronger than the corresponding one, i.e. $W(\nabla f,.,.,\nabla f)=0$, in \cite[Theorem 7.1]{HPW} and in Theorem \ref{improved 7.1}. Also, the condition $p\in O$ in \cite[Theorem 7.1]{HPW} is stronger than the condition $\nabla f|_p\neq0$ in Theorem \ref{improved 7.1}. Therefore, Theorem \ref{improved 7.1} is a stronger and more general result.
\begin{mylem}
\label{u and f r}
Let $m>1$, $\lambda\in\mathbb{R}$ and $(M,g,f)$ satisfies Equation \eqref{hess} of a nontrivial $(\lambda,n+m)$-Einstein manifold with harmonic Weyl tensor and $(\nabla f,.,.,\nabla f)=0$. In addition, suppose that for a coordinate system $(t,t_1,t_2,...,t_{n-1})$ the metric is of the form 
\begin{align}
g=dt^2+u^2(t)g_N,\enskip t\in[-\epsilon,\epsilon],\enskip\epsilon\in\mathbb{R}^+\label{warped product r}
\end{align}
where $u(t)>0$ on $(-\epsilon,\epsilon)$ 
and the function satisfies
\begin{align}
 f(t,x)=f(t),\enskip x\in N.\label{f formul}
\end{align}
If $u(a)=0$, $a\in\{-\epsilon,\epsilon\}$, then $f'(a)=0$.   
\end{mylem}
\begin{proof}
By contradiction suppose that $f'(a)\neq0$. Then by the proof of Theorem \ref{improved 7.1} we see that the warped product is extendable to a neighborhood of $t=a$, hence the warping function must satisfy $u(a)>0$, which is a contradiction to the assumption $u(a)=0$.
\end{proof}
\begin{mycor}
\label{as far as} 
Under the same assumptions as Theorem \ref{improved 7.1}.1) we have $u(t)>0$ if and only if $f'(t)\neq0$ (equivalently $u(t_0)=0$ if and only if $f'(t_0)=0$ where $t_0$ is the first zero for them ). 
\end{mycor}
\begin{proof}
By combination of Theorem \ref{improved 7.1}, Lemma \ref{u and f r} the result follows.
\end{proof}
\section{Local characterization of Riemannian $(\lambda,n+m)$-Einstein metrics $(M,g,f)$ around critical points of $f$}
In the Riemannian setting, when using geodesic polar coordinates $(t,x)\in\mathbb{R}\times S(1)$ we consider local warped product metrics of the form $dt^2+u^2(t)g_1(x)$, $(t,x)\subset\mathbb{R}\times S^{n-1}(1)$, where $g_1$ is the induced metric on the standard sphere $S^{n-1}(1)$. Hence at critical points of $f$ being located at the origin of local geodesic polar coordinates the warping function vanishes, i.e. $u(0)=0$.

In this subsection we first show that under the assumptions of harmonicity of the Weyl tensor and $W(\nabla f,.,.,\nabla f)=0$, critical points of $f$ in a Riemannian $(\lambda,n+m)$-Einstein metric, in geodesic polar coordinates with origin located at the critical points, are isolated and the level sets close to critical points are isometric to spheres. Around critical points of $f$, under some conditions on the warping function of a given warped product the metric is conformally flat. Via these properties we then locally characterize non-trivial $(\lambda,n+m)$-Einstein manifolds $(M,g,f)$ around critical points of $f$.
\begin{mylem}
\label{isolated critical}
Let $m>1$, $\lambda\in\mathbb{R}$ and $(M,g)$ be a connected Riemannian manifold with a smooth function $f$ defined on $M$. Assume that $(M,g,f)$ satisfies Equation \eqref{hess} of a non-trivial $(\lambda,n+m)$-Einstein metric with harmonic Weyl tensor satisfying $W(\nabla f,.,.,\nabla f)=0$ in a neighborhood of $p\in M$ with $\nabla f|_p=0$. Then there exists a neighborhood $\mathcal{U}$ of $p$ such that
\\ 

(i) $p$ is the only critical point of $f$ in $\mathcal{U}$,\\

(ii) The level hypersurfaces of $f$ in $\mathcal{U}$ coincide with the geodesic distance spheres around $p$.
\end{mylem}
\begin{proof}
At first we consider geodesic polar coordinates with origin at $p$. We choose $\mathcal{U}$ such that every point in $\mathcal{U}$ has a unique and shortest geodesic joining it with $p$. Then we consider $q$ to be a regular point of $f$. By Remark \ref{p6} we have  $f(p)\neq f(q)$. Let  $\mathcal{A}:=\{\bar{q}\in\mathcal{U}|f(\bar{q})=f(q)\}$ consisting of only regular points.  Consider the trajectory geodesics of $\nabla f$ starting at the hypersurface $\mathcal{A}$  and pointing to one of its sides, without loss of generality to the side containing $p$. $\mathcal{A}$ contains a point $q_0$ realizing the distance
 $$s_0:=d(p,q_0)=d(p,\mathcal{A})>0.$$
EITHER there is a minimizing geodesic $\gamma_0$ joining $p$ and $q_0$ consisting only of regular points of $f$ (the case considering the possibility of existing a critical point of $f$ between $p$ and $q_0$ along $\gamma_0$ will be investigated in the rest of this proof). This realizes the distance between $\gamma_0(0)=p$ and $\gamma_0(s_0)=q_0$. By the Gauss lemma $\gamma_0$ meets $\mathcal{A}$ perpendicularly. Consequently by a discussion in the proof of Theorem \ref{improved 7.1} $\gamma_0$ is the same curve (up to parameterization) as the trajectory of $\nabla f$ through $q_0$. Any other point $q_1\in\mathcal{A}$ yields similarly a geodesic trajectory $\gamma_1$ of $\nabla f$. Let $\gamma_1(s_0)=q_1$. Then the claim is that $\gamma_1(0)=p$. To see this let $d_M$ and $d_{\mathcal{A}(s)}$ denote the distance functions in $M$ and the level set $\mathcal{A}(s)$ corresponding to the parameter $s$, respectively. Then for any $s>0$
\begin{align}
 d_M(\gamma_0(s), \gamma_1(s))&\leq d_{\mathcal{A}(s)}(\gamma_0(s), \gamma_1(s))\nonumber\\
 &=\frac{u(s)}{u(s_0)}d_{\mathcal{A}(s_0)}(\gamma_0(s_0), \gamma_1(s_0)).\label{u0}
\end{align}
For the last equality in \eqref{u0} we used the warped product metric according to Theorem \ref{improved 7.1}. Since the critical point $p$ is located at the origin of geodesic polar coordinates we have $u(0)=0$. It follows then
\begin{align}
 d_M(\gamma_0(0),\gamma_1(0))&=\lim_{s\to 0}d_M(\gamma_0(s),\gamma_1(s))\nonumber\\
 &\leq\lim_{s\to 0}\frac{u(s)}{u(s_0)} d_{\mathcal{A}}(\gamma_0(s_0),\gamma_1(s_0))=0\label{equal to 0}.
\end{align}
Therefore, $\gamma_1(0)=\gamma_0(0)=p$, and $\mathcal{A}$ is contained in the geodesic distance sphere with radius $s_0$ around $p$. On the other hand it follows that the arc length parameter on the trajectories is just the geodesic distance to $p$. Therefore $p$ is the only critical point in $\mathcal{U}$, and the $f$-levels coincide there with the geodesic distance spheres around $p$.\\\\
OR the same argument as above shows that  in a certain minimal distance $s_1$ $(<s_0)$ there are critical points, the same distance on each trajectory, and ultimately all trajectories pass through the same critical point $p_1$. Then there are only regular points between $p_1$ and $q$. This implies that in an open neighborhood $p_1$ is surrounded by non-critical level sets of $f$ (all diffeomorphic with the $(n-1)$-sphere), so this critical point $p_1$ is also isolated. But by connectedness there can not be two critical points $p$ and $q_1$ at the same side of $\mathcal{A}$.
\end{proof}
\begin{mycor}
\label{at most two}
 The same procedure on the other side of $\mathcal{A}$ shows that either there is no critical point or there is precisely one other critical point $p'$ with the same properties. In combination this seems to show that three or more critical points are impossible if $\mathcal{A}$ is connected (compare Theorem \ref{aroundcritical} and Theorem \ref{Petersen22.2}). If $\mathcal{A}$ is not connected one has the same situation for each component separately.
\end{mycor}
\begin{mytheo}
\label{aroundcritical}
 Let $m>1$, $\lambda\in\mathbb{R}$ and $(M,g)$ be a Riemannian manifold with a smooth function $f$ defined on $M$. Then the following conditions are equivalent:\\
 
 1) $(M,g,f)$ is conformally flat and satisfies Equation \eqref{hess} of a non-trivial $(\lambda$,n+m)-Einstein metric in a neighborhood of $p\in M$ with $\nabla f|_p=0$.\\
 
 2) There exist polar coordinates 
 $(t,t_1,...,t_{n-1})\in I\times S^{n-1}(1)$, $0\in I\subseteq\mathbb{R}$ being an open interval, in a neighborhood of 
 $p$ and an odd function $u=u(t)$, i.e. $u(0)=u^{(even)}(0)=0$, with $u(t)>0$ on $t\in I-\{0\}$ and $0\neq (u')^2(0)=k$, such that in these coordinates $f=f(t)$ and  
\begin{align}
 g=dt^2+\frac{u^2(t)}{k}g_{S^{n-1}(1)}\label{criticalchar}
\end{align}
where $g_{S^{n-1}(1)}$ denotes the line element of the standard unit sphere $S^{n-1}(1)$; In addition, the conditions \eqref{equ1} and \eqref{equuu} in Proposition \ref{equ} hold.
\end{mytheo}
\begin{proof}
1) $\Rightarrow$ 2) We consider a neighborhood $\mathcal{U}$ of $p$ such that $g|_{\mathcal{U}}$ be conformally flat. This provides $g$ with harmonic Weyl tensor and the property $W(\nabla f,.,.,\nabla f)=0$ on $\mathcal{U}$. Hence, by Lemma  \ref{isolated critical} we already know that (after restriction of $\mathcal{U}$ if necessary) $p$ is the only critical point of $f$ in $\mathcal{U}$. In addition, by Theorem \ref{improved 7.1} and Lemma \ref{isolated critical} we may introduce locally coordinates such that for $t\neq0$  
\begin{align}
g=dt^2+u^2(t)g_N\nonumber\\
f(t,x)=f(t),\enskip x\in N\label{bysmooth}
\end{align}
where $g_N$ is the induced metric on a regular level set $N$ of $f$. By smoothness of $f$ and $g$ it follows that the equations in \eqref{bysmooth} hold for $f$ and $g$ at the time $t=0$ as well.  
 
Let $X,Y$ be two orthonormal vectors in $M$ which are tangent to a level hypersurface $N=\{q|f(q)=t_0>0\}$ for sufficiently small $t_0$. By Equation \cite[Lemma 2.5.(16)]{KR3} the sectional curvatures $Sec$ resp. $Sec_N$ of the $(X,Y)$-plane in $(M,g)$ resp. $(N,g_N)$ satisfy
\begin{align*}
Sec&=g(R(X,Y)Y,X)\\
 &=g(R_N(X,Y)Y,X)-\frac{(u'(t_0))^2}{(u(t_0))^2}\\
 &=\frac{1}{(u^2(t_0))^2}(Sec_N-(u'(t_0))^2).
\end{align*}
 On the other hand, $g_N$ is independent of $t$ when it tends to zero, cf. the proof of Theorem \ref{improved 7.1}. Since $u(0)=0$, due to the fact that the critical point $p$ is located on the origin, it follows that  
$$0=\lim_{t\to0}(Sec_N-(u'(t))^2)=Sec_N-(u'(0))^2.$$
It implies that $(N,g_N)$ is a space of constant curvature $Sec_N=(u'(0))^2$, and hence either $Sec_N>0$ or $Sec_N=0$. We already know that $N$ is a geodesic distance sphere which is diffeomorphic to $S^{n-1}$, hence the case $Sec_N=0$ can not occur. Because in this case
$N$ will be flat with Euclidean space as its universal cover, while we know that the universal cover of any sphere is itself. Therefore $Sec_N>0$.

Consequently $(u'(0))^2>0$ and $g_N=\frac{1}{(u'(0))^2}g_{S^{n-1}(1)}$. Moreover using Proposition \ref{equ} via multiplying both sides of Equation \eqref{equ32} with $u^2(t)$ and then taking the limit while $t$ tends to 0 as well as noting Proposition
\ref{constantboundary}, which implies $f(0)\neq0$, one obtains $u'^2(0)=k$. Hence $g_N=\frac{1}{k}g_{S^{n-1}(1)}$. By our assumption the metric $g$ is everywhere smooth and has no singularity at $t=0$. Therefore using the same calculation as \cite[1.4.4]{P} we conclude that $u(t)$ is an odd function at $t=0$, i.e. $u^{(even)}(0)=0$ and Equation \eqref{criticalchar} is valid for all $t\geq0$ as the usual expression of the Euclidean metric in polar coordinates. Since $u(t)$ and $f(t)$ are continuous and since by assumption $(M,g,f)$ satisfies Equation \eqref{hess}, by Proposition \ref{equ}  the conditions \eqref{equ1} and \eqref{equuu} are satisfied.\\

2) $\Rightarrow$ 1) In order to see that $\nabla f|_p=0$ we use the assumption $u(0)=0$ together with Lemma \ref{u and f r}. In order to see that Equation \eqref{criticalchar} together with $f=f(t)$, $t\in I$, satisfies Equation \eqref{hess} of a $(\lambda,n+m)$-Einstein manifold in polar coordinates, one may apply Theorem \ref{improved 7.1} for all points except $t=0$. The oddness of the function $u(t)$, i.e. $u^{(even)}(0)=0$, and $0\neq u'^2(0)=k$ yields that the right hand side of \eqref{criticalchar} has no proper singularity at $t=0$. Thus by continuity Equation \eqref{hess} holds at $t=0$ as well. Moreover, since by assumption the function $u(t)$ in  \eqref{criticalchar} satisfies $u'(0)\neq0$ and $u^{(even)}(0)=0$, via similar calculations to the proof of \cite[Proposition 3.5]{KR3} we see that the local warped metric \eqref{criticalchar} is conformally flat. 
\end{proof}
Reminder: In Theorem \ref{aroundcritical}, if in addition $(M,g)$ is Einstein then Hess$f$ would be proportional to the metric $g$. In this situation there is already a characterization by \cite[Lemma 18]{KW} in terms of a local warped decomposition of $g$ with $f'(t)$ as the warping function. 
\section{Global characterization and some tips on Theorem \ref{correction}}
There are some points to be discussed on Theorem \ref{correction}. The first point: In Theorem \ref{correction} the authors assume the manifold is simply connected, while through the next result we see that ``simply connected'' is not needed in the formulation of the theorem. 

The second point: Under the assumptions of the theorem the number of critical points of $f$ can be at most two, cf. Corollary \ref{at most two} or Theorem \ref{Petersen22.2}. In particular the warped product structure is global, i.e. complete, if there are no critical points for $f$, see Corollary \ref{Petersen2}. 

The third point: It is necessary to show that the critical points of $f$ are isolated, cf. Lemma \ref{isolated critical}. Additionally, in order that the local warped product $g=dt^2+u^2(t)g_N$ on $M-\{\text{critical points of}$ $f\}$ can be extended smoothly to a metric on $M$ around critical ponts of $f$  we need that the warping function $u(t)$ be odd on the critical points of $f$, i.e. $u^{(even)}(\gamma_0)=0$ where $\gamma_0=t(q)$ with $\nabla f(q)=0$, as well as $u'(\gamma_0)\neq0$, cf. Theorem \ref{aroundcritical}. 

The fourth point: Although the relations between $u(t)$ and $f(t)$ are investigated for some specific cases such as $(0,n+m)$-Einstein and $(\lambda,2+m)$-Einstein metrics in \cite{HPW}, but for a   
$(\lambda,n+m)$-Einstein metric they are not in general formulated. To generally relate $u(t)$ and $f(t)$ in  the formulation of the theorem we may use \eqref{equ1}, \eqref{equuu} and \eqref{equ0} of Proposition \ref{equ}.  

The fifth point: By \cite[16.26(i)]{B} we may include the properties harmonicity of the Weyl tensor of $(M,g)$ and $W(\nabla f,.,.,\nabla f)=0$ in the characterization equivalence. In other words, these two properties may be moved from being part of the assumption into the equivalence relation of the characterization. This is due to \cite[16.26(i)]{B} which says that any warped product $g=dt^2+u^2(t)g_N$ with Einstein fibre $g_N$ has harmonic Weyl tensor and satisfies $W(\nabla f,.,.,\nabla f)=0$.  

Now considering all these points together we can restate the global statement \cite[Theorem 7.2]{HPW} as follows. Here $]($ means either $]$ or $($. 
\begin{mytheo}
\label{Petersen22.2}
Let $m>1$, $\lambda\in\mathbb{R}$, $(M,g)$ a connected Riemannian manifold and $f$ be a smooth function on $M$. Then the following conditions are equivalent:\\

1) $(M,g,f)$ is a non-trivial $(\lambda,n+m)$-Einstein metric with harmonic Weyl tensor and $W(\nabla f,.,.,\nabla f)=0$.\\

2) If we let $C:=\{q\in M: \nabla f(q)=0\}$, then $|C|\leq2$ and $(M\setminus C,g)$ is isometric with a warped product metric 
 \begin{align}
 &g=dt^2+u^2(t)g_N,\label{warpeded}\\
 &f=f(t)\label{ft}
 \end{align}
on $I\times N$ where $g_N=g_{S^{n-1}}$ and $g$ is conformally flat if $|C|>0$, otherwise $(N,g_N)$ is a complete Einstein Riemannian hypersurface of $(M,g)$ say with normalized scalar curvature $\varrho_N=k$. $I=[\alpha_0,\beta_0)]$ with $\alpha_0\in\mathbb{R}$ if there exists a point $q_0\in\partial M$ satisfying $f(q_0)=f(\alpha_0)=0$. $I=(\alpha_0,\beta_0)]$, $\alpha_0\in\mathbb{R}$ only if there exists a minimum point $q_0$ of $f$ satisfying  $f(q_0)=f(\alpha_0)$, and hence, $u=u(t)$ must be odd at $\alpha_0$, i.e. $u^{(even)}(\alpha_0)=0$, with $u'(\alpha_0)\neq0$. The same discussion holds for the right side of $I$ regarding the boundary of $M$ and critical points of $f$, except that when there is a critical point for $f$ then it must correspond to a maximum point of $f$. If neither of the cases happens for the left and right sides of $I$, then $I = (-\infty,\infty)$.\\  
The product $I\times N$ becomes complete if we add the set $C$ of critical points to it. In addition, $f=f(t)$ and $u=u(t)$ satisfy the equations \eqref{equ1}, \eqref{equuu} and \eqref{equ0} in Proposition \ref{equ}.
\end{mytheo}
\begin{proof}
1) $\Rightarrow$ 2): By Lemma \ref{isolated critical}, $C$ is a set of isolated points. For every fixed point $q\in M\setminus C$, by Theorem \ref{improved 7.1} there is an open neighborhood $\mathcal{U}\ni q$ in which equations \eqref{warpeded} and \eqref{ft} hold. Where $g_N$, $N := \{x\in M|f(x)=f(q)\}$, is an Einstein Riemannian hypersurface of $(M,g)$ say with normalized scalar curvature 
$k=\varrho_N$. The hypersurface $(N,g_N)$ is complete as every Cauchy sequence in $N$ converges in $M$. Accordingly we have $\mathcal{U}=(\alpha,\beta)\times N$. The trajectory through $q$ is the unique geodesic with tangent $\frac{\partial}{\partial t}$. By completeness this is defined for every parameter $t$ as far as it does not hit the boundary $\partial M$. 

We define $\alpha_0$ and $\beta_0$ to be the infimum and supremum of $\alpha$, $\beta$ such that \eqref{warpeded} holds for $(\alpha,\beta)\times N$. Here the extension to $(\alpha_0,\beta_0)\times N$ is  
regardless of whether or not the points belong to the set $O$. In fact, it continues as long as the points lie in regular level sets of $f$, cf. Theorem \ref{improved 7.1}. Moreover, a similar discussion as in Theorem \ref{improved 7.1} implies that $f(t,x)=f(t)$ on $(\alpha_0, \beta_0)\times N$ where $x\in N$. 

If $\alpha_0$ (or $\beta_0$) is finite then there is a limit point $q_0$ on this geodesic with $f(q_0)=f(\alpha_0)$ (or $f(\beta_0)$). If $q_0$ is boundary point i.e. $f(q_0)=f(\alpha_0)=0$ (or $f(q_0)=f(\beta_0)=0$) by Proposition \ref{constantboundary} it follows that $\nabla f(q_0)\neq0$. Thus Equation \eqref{distance0} in the proof of Proposition \ref{improved 7.1} implies that $t$ would be also smooth at  $\alpha_0$ (or $\beta_0$). Therefore by smoothness of $g$, $f$ and $t$ equations \eqref{warpeded} and \eqref{ft} are valid at $\alpha_0$ (or $\beta_0$) as well. In addition, since $\{\alpha_0\}\times N$ (or $\{\beta_0\}\times N$) is a component of $\partial M$ by completeness we have $f^{-1}((-\infty,\beta_0)])=f^{-1}((\alpha_0,\beta_0)])$ (or $f^{-1}([(\alpha_0,\infty))=f^{-1}([(\alpha_0,\beta_0))$). 

If $q_0$ is not a boundary point then it must be a critical point of $f$ because otherwise by the argument in Theorem \ref{improved 7.1} Equation \eqref{warpeded} could be extended to a neighborhood of $q_0$ which is a contradiction. Furthermore $q_0$ is a minimum (a maximum) of $f$, because by Theorem \ref{aroundcritical} $u'(q_0)\neq0$ and hence Hess $f$ is definite at $q_0$. By connectedness no other critical points can occur, cf. the proof of Lemma \ref{isolated critical}. Hence the number of critical points of $f$ is at most two, i.e. $|C|<2$; compare also Corollary \ref{at most two}. Moreover, by our assumption the metric $g$ and $f$ are everywhere smooth and have no singularity at $\gamma_0=\alpha_0$. Hence through the same calculation as in \cite[1.4.4]{P} it follows that $u(t)$ is an odd function at $\gamma_0=\alpha_0$ (or $\beta_0$), i.e. $u^{(even)}(\gamma_0)=0$, where $\gamma_0\in\mathbb{R}$ and $u'(\gamma_0)\neq0$ and moreover equations \eqref{warpeded} and \eqref{ft} are valid for all $t\in(\alpha_0,\beta_0)$ and for $\alpha_0\in\mathbb{R}$ (or $\beta_0\in\mathbb{R}$).

$I\times N$ becomes complete if we add the set $C$ of critical points of $f$ to it. More explicitly, when $|C|=0$ then $g$ would be global, i.e. complete. If $|C|=1$ then $g$ would be complete through addition the lone critical point which is the minimum of level $f(\alpha_0)$ (or the maximum of level $f(\beta_0)$). If $|C|=2$ then $g$ becomes complete by adding both the minimum of level $f(\alpha_0)$ and the maximum of level  $f(\beta_0)$.

Moreover, due to smoothness of $g$, $f$ and $t$ as well as oddness of $u(t)$ at finite $\gamma_0=\alpha_0$ (or $\gamma_0=\beta_0$), i.e. $u^{(even)}(\gamma_0)=0$,  the functions $u=u(t)$ and $f=f(t)$ satisfy the equations \eqref{equ1},  \eqref{equuu} and \eqref{equ0} in Proposition \ref{equ}.\\

2) $\Rightarrow$ 1): By assumption $(M\setminus C,g)$ is isometric with $dt^2+u^2(t)g_N$ and $f(t,x)=f(t)$ on $I\times N$; $x\in N$ where $I$ and $N$ are as described. Also, by assumption even derivatives of $u(t)$ vanish, i.e. $u(t)$ is odd, at finite end points of the interval $I$ when they correspond to a critical point of $f$. Therefore, the warped product metric extends smoothly  to a metric on $M$. 

Since by assumption $I\times N$ becomes complete after addition with the set $C$ of critical points of $f$, we conclude that the metric $g$ is complete as well.

By assumption, $u(t)$ and $f(t)$ satisfy equations \eqref{equ1},  \eqref{equuu} and \eqref{equ0} in Proposition \ref{equ} on $I$. On the other hand $u(t)$ is odd on the critical points of $f$. Therefore, by Proposition \ref{equ} it follows that $(M,g,f)$ is a $(\lambda,n+m)$-Einstein manifold. 

As $g_N$ is Einstein by \cite[16.26(i)]{B}, together with the smoothness of the metric in \eqref{warpeded} and oddness of $u(t)$ at critical points of $f$, we conclude that $(M,g)$ has harmonic Weyl tensor and satisfies $W(\nabla f,.,.,\nabla f)=0$. 
\end{proof}
\begin{mycor}
	\label{Petersen2}
	Let $m>1$, $\lambda\in\mathbb{R}$ and $(M,g)$ be a connected Riemannian manifold (without boundary) on which a smooth function $f$ is defined. Then the following conditions are equivalent:\\
	
	1) $(M,g,f)$ is a non-trivial $(\lambda,n+m)$-Einstein metric with harmonic Weyl tensor and $W(\nabla f,.,.,\nabla f)=0$ and where $f$ has no critical point.\\
	
	2) $g=dt^2+u^2(t)g_N$ on $\mathbb{R}\times N$ where $(N,g_N)$ is a complete Einstein Riemannain hypersurface of $(M,g)$ say with normalized scalar curvature $\varrho_N=k$, in addition $f=f(t):\mathbb{R}\rightarrow\mathbb{R}^+$ which together with $u=u(t):\mathbb{R}\rightarrow\mathbb{R}^+$ satisfies \eqref{equ1} and \eqref{equuu} in Proposition \ref{equ}.
\end{mycor}
\bibliography{bib}
\bibliographystyle{amsalpha}
\end{document}